\documentclass[a4paper,english,nostix]{birthday}
\usepackage[latin1]{inputenc}
\usepackage{amsmath,babel}
\usepackage{undertilde}
\usepackage{amsthm}
\usepackage{amstext}

\usepackage{slashed}

\usepackage{amsmath,amssymb,amsfonts,epsfig,mathrsfs}

\newtheorem{definition}{Definition}[section]
\newtheorem{theorem}[definition]{Theorem}
\usepackage{amscd,psfrag}
\usepackage{yhmath}
\usepackage[mathscr]{eucal}

\newtheorem{lemma}[definition]{Lemma}
\newtheorem{corollary}[definition]{Corollary}

\newtheorem{proposition}[definition]{Proposition}

\newcommand{\real}{\mathbb{R}}
\newcommand{\R}{\mathbb{R}}
\newcommand{\vf}{\Gamma(TM)}
\newcommand{\na}{\nabla}

\newcommand{\rnk}{\mathbb{R}^{n+k}}

\newcommand{\qtens}{\bigwedge^q T^*M}

\newcommand{\nap}{\nabla^\perp}
\newcommand{\bep}{B^\epsilon}
\newcommand{\nep}{\nabla^{\perp,\epsilon}}

\newcommand{\vb}{V^{(B)}}
\newcommand{\vnab}{V^{(\nabla^\perp)}}
\newcommand{\ob}{\Omega^{(B)}}
\newcommand{\onab}{\Omega^{(\nabla^\perp)}}
\newcommand{\vbe}{V^{(B^\epsilon)}}
\newcommand{\vnabe}{V^{(\nabla^{\perp,\epsilon})}}
\newcommand{\obe}{\Omega^{(B^\epsilon)}}
\newcommand{\onabe}{\Omega^{(\nabla^{\perp,\epsilon})}}

\newcommand{\woneploc}{W^{1,p}_{\rm loc}}

\newcommand{\lploc}{L^p_{\rm loc}}
\newcommand{\hil}{\mathcal{H}}
\newcommand{\sdag}{S^\dagger}
\newcommand{\tdag}{T^\dagger}

\newcommand{\st}{S\oplus T}
\newcommand{\stdag}{\sdag\vee\tdag}
\newcommand{\useq}{\{u^\epsilon\}}
\newcommand{\vseq}{\{v^\epsilon\}}
\newcommand{\ub}{\bar{u}}
\newcommand{\vbb}{\bar{v}}
\newcommand{\ran}{\text{ran}}

\newcommand{\antil}{\tilde{a}^\epsilon}

\newcommand{\bntil}{\tilde{b}^\epsilon}

\newcommand{\yz}{Y\bigoplus Z}
\newcommand{\yzstar}{Y^*\bigoplus Z^*}
\newcommand{\sd}{\slashed{\Delta}}

\newcommand{\htt}{\utilde{\hil}}

\newcommand{\e}{\epsilon}

\newcommand{\loc}{\text{loc}}
\newcommand{\p}{\partial}

\newcommand{\weak}{\rightharpoonup}
\newcommand{\di}{{\rm div}}
\newcommand{\curl}{{\rm curl}}
\newcommand{\emb}{\hookrightarrow}

\numberwithin{equation}{section}

\title{Compensated Compactness in Banach Spaces and \\Weak Rigidity of Isometric Immersions of Manifolds}
\author{Gui-Qiang G. Chen and Siran Li}
\contact[chengq@maths.ox.ac.uk]{Mathematical Institute, University of Oxford, Oxford, OX2 6GG, UK}
\contact[Siran.Li@rice.edu]{Mathematical Institute, University of Oxford, Oxford, OX2 6GG, UK}

\begin{document}
\begin{center}
 {\it To Helge Holden on the occasion of his $60$th birthday\\ with friendship and affection}
\end{center}

\begin{abstract}
We present a compensated compactness theorem in Banach spaces established recently,
whose formulation is originally motivated by the weak rigidity problem
for isometric immersions of manifolds with lower regularity.
As a corollary, a geometrically intrinsic div-curl lemma
for tensor fields on Riemannian manifolds is obtained.
Then we show how this intrinsic div-curl lemma can be employed
to establish the global weak rigidity of the Gauss-Codazzi-Ricci equations,
the Cartan formalism, and the corresponding isometric immersions of Riemannian submanifolds.
\end{abstract}
\begin{classification}
{Primary:
53C24, 53C42, 53C21, 53C45, 57R42, 35M30, 35B35, 58A15, 58J10;
Secondary: 57R40, 58A14, 58A17, 58A05, 58K30, 58Z05}.
\end{classification}

\begin{keywords}
Compensated compactness, weak rigidity, global, intrinsic,
Gauss-Codazzi-Ricci equations,
Riemannian manifolds, isometric immersions,
isometric embeddings, lower regularity, weak convergence,
approximate solutions, div-curl lemma,
Riemann curvature.
\end{keywords}

\section{Introduction}

In this paper we discuss a unified approach developed recently in \cite{chenli}
towards establishing more general and intrinsic compensated compactness theorems
for nonlinear analysis and nonlinear partial differential equations (PDEs),
with applications to the weak rigidity of isometric immersions
of Riemannian manifolds into Euclidean spaces with lower regularity.

Compensated compactness has played an important role
in the study of nonlinear PDEs arising from fluid mechanics,
calculus of variations,
and nonlinear elasticity;
{\it cf.} \cite{Ball, CLMS, Dacorogna, Dafermos-book, Evans-Muller, Mur78, Tar79, Tartar2}
and the references cited therein.
The {\em div-curl lemma} introduced by Murat and Tartar \cite{Mur78,Tar79}
is the cornerstone of the theory, which reads in the simplest and original form as follows:

\begin{lemma}[The div-curl lemma by Murat-Tartar \cite{Mur78,Tar79}]\label{proposition: murat-tartar}
Let $\{v^\e\}, \{w^\e\} \subset L^2_{\loc}(\R^3; \R^3)$ be two sequences of vector fields on $\R^3$
such that
$$
v^\e \weak \bar{v}, \quad w^\e \weak \bar{w} \qquad\,\, \mbox{weakly in $L^2_\loc$}.
$$
Assume that $\{\di\, v^\e\}$ is pre-compact in $H^{-1}_{\loc}(\R^3; \R)$
and $\{\curl\, w^\e\}$ is pre-compact in $H^{-1}_\loc(\R^3; \R^3)$.
Then
$$
v^\e \cdot w^\e \, \to \, \bar{v}\cdot \bar{w} \qquad\quad
\mbox{in the sense of distributions.}
$$
\end{lemma}

Two distinctive approaches have been developed in the literature
to prove Lemma \ref{proposition: murat-tartar}: One is via harmonic analysis,
and the other is based on the Hodge (de Rham) decomposition theorem;
see \cite{CLMS, Eva90, KY13, Murat2, RRT87} and the references cited therein.
Both approaches depend crucially on the geometry of the Euclidean spaces
or Riemannian manifolds; see \S 2 for a detailed exposition.
See also Whitney \cite{Whitney57} (Chapter IX, Theorem 17A)
for an early version of the div-curl type lemma in the language of his
{\it geometric integration theory}.

One of our crucial observations in \cite{chenli} is that the div-curl lemma can be reformulated
via a functional-analytic approach in generic Banach spaces with two general bounded linear operators
in place of {\em div} and {\em curl}.
Indeed, all the necessary properties we need for {\em div} and {\em curl} in order to conclude the lemma
can be extracted as two simple, abstract conditions: One is algebraic, and the other is analytic.
Both conditions can be naturally formulated in terms of operator algebras on Banach spaces.
This leads to the
generalization of the existing versions of the div-curl lemma,
and provides the third approach to the theory of compensated compactness.
In addition, combining the functional-analytic compensated compactness theorem together
with the ellipticity of the Laplace-Beltrami operator, we are ready to obtain
a geometrically intrinsic div-curl lemma on Riemannian manifolds.
Throughout this paper, the term ``{\it intrinsic}''
means ``{\it independent of local coordinates}'' on Riemannian manifolds.

As an application of the
new div-curl lemma, we analyze the weak rigidity of isometric immersions
of Riemannian manifolds into Euclidean spaces.
The problem of isometric immersions/embeddings has been of considerable interest
in the development of differential geometry, which has also led to the important developments
of new ideas and methods in nonlinear analysis and PDEs ({\it cf.} \cite{HanHon06, Nas54, Nas56, yau}).
Moreover, it is well-known in differential geometry  that
the Gauss-Codazzi-Ricci (GCR) equations are the compatibility conditions for the existence
of isometric immersions (see \cite{DoC92, Spi79}).
The GCR equations can be viewed as a first-order nonlinear system of geometric PDEs for the
second fundamental forms and normal connections. However, in general, the GCR equations
are of no type, neither purely hyperbolic nor purely elliptic.

The weak rigidity problem for isometric immersions can be formulated as follows:
Given a sequence of isometric immersions of an $n$-dimensional manifold with
a $W^{1,p}_{loc}$  metric for $p>n$,
whose  second fundamental forms and normal connections are uniformly bounded
in $L^p_{\rm loc}$,
whether its weak limit is still an isometric immersion with
the same $W^{1,p}_{\rm loc}$ metric.
This rigidity problem has its motivation from both geometric analysis
and nonlinear elasticity:
The existence of isometric immersions of Riemannian manifolds with lower regularity
corresponds naturally to the realization of elastic bodies with lower regularity
in the physical space.
See Ciarlet-Gratie-Mardare \cite{Cia08}, Mardare \cite{Mar05}, Szopos \cite{Szo08},
and the references cited therein.

In \cite{chenli}, we have proved that the solvability of the GCR equations in $W^{1,p}$
is equivalent to the existence of $W^{2,p}$ isometric immersions on Riemannian manifolds.
This is done by employing the Cartan formalism, also known as the method of moving frames.
We have shown that both the GCR equations and isometric immersions are equivalent
to the structural equations of the Cartan formalism.
Then, by exploiting the {\em div-curl structure} of the GCR equations
and the Cartan formalism,
we have deduced the global weak rigidity of these geometric PDEs,
independent of local coordinates
on Riemannian manifolds.
Now, in view of the equivalence theorem established above,
the weak rigidity of isometric immersions is readily concluded.

The rest of this paper is organized as follows:
In \S $2$, we first formulate
the functional-analytic compensated compactness theorem in Banach spaces
and give an outline of its proof,
and then deduce a geometrically intrinsic div-curl lemma on Romannian manifolds
as its corollary.
Two generalizations of the latter result are also discussed.
In \S 3, we collect some background on differential geometry pertaining
to the GCR equations
and the Cartan formalism.
Finally, in \S 4, we show the weak rigidity of isometric immersions,
together with the weak rigidity of the GCR equations and the Cartan formalism.

\section[\small A Compensated Compactness Theorem in Banach Spaces]{A Compensated Compactness Theorem in Banach Spaces}

In this section we first discuss a functional-analytic compensated compactness theorem. As its consequence, we deduce
a geometrically intrinsic div-curl lemma on Riemannian manifolds.

To establish the original div-curl lemma, Lemma \ref{proposition: murat-tartar}, as well as its various
generalizations (see \cite{CLMS, Eva90, KY13, Murat2, RRT87} and the references cited therein),
the following distinctive approaches have been adopted:
	
The first approach, developed by Murat and Tartar in \cite{Mur78, Tar79}, is based on {\em harmonic analysis}.
It is observed that the first-order differential constraints, namely the pre-compactness of $\{\di\,v^\e\}$
and $\{\curl\, w^\e\}$ in $H^{-1}$, lead to the decay properties of $\{v^\e\cdot w^\e\}$
in the high Fourier frequency region.
Coifman-Lions-Meyer-Semmes in \cite {CLMS}
extended this lemma by combining the exploitation of this observation with
further techniques in harmonic analysis, including Hardy spaces,
and commutator estimates of BMO functions and Riesz transforms.

The second approach is based on the {\em Hodge decomposition}.
Robbin-Rogers-Temple in \cite{RRT87}
observed that, by writing
\begin{equation}\label{eq: hodge in R3}
v^\e = \Delta \Delta^{-1} v^\e = ({\rm grad} \circ \di - \curl \circ \curl)\Delta^{-1} v^\e,
\end{equation}
$\{v^\e\}$ can be decomposed into a weakly convergent part and a strongly convergent part, and similarly for $w^\e$
(also see the exposition in Evans \cite{Eva90}).
For this,
the advantage can be taken of
the first-order differential constraints,
the commutativity of the Green operator $\Delta^{-1}$ on $\R^3$ with divergence,
gradient, and curl, and most crucially, the ellipticity of $\Delta$,
so that, for $\{v^\e\cdot w^\e\}$, the pairing of the weakly convergent terms
pass to the limits via integration by parts,
and the pairings of other terms can be dealt with directly.
Observe that the Laplace-Beltrami operator $\Delta$ defined for differential forms on any oriented closed
Riemannian manifold $(M,g)$ is always elliptic, and it has a  decomposition similar to \eqref{eq: hodge in R3}:
\begin{equation}\label{eq: hodge for any mfd}
\Delta = d\circ \delta + \delta \circ d,
\end{equation}
where $d$ is the exterior differential and $\delta$ is
its $L^2$-adjoint ({\it cf.} \S 6 in \cite{War71} for the details),
so that the div-curl lemma is ready to be generalized to Riemannian manifolds.

\smallskip	
The third approach, which is the main content of this section,
is {\em functional-analytic}.
As aforementioned, the existing div-curl lemmas are formulated in terms of vector
fields or local differential forms on Euclidean spaces ({\it cf.} \cite{CLMS, Eva90, RRT87, Mur78, Tar79}),
and some generalizations to Riemannian manifolds are available ({\it cf.} \cite{chenli, I, KY13}).
For example, Kozono-Yanagisawa \cite{KY13} obtained a
div-curl lemma using functional-analytic results on $L^2(\R^n)$,
as well as a geometric version,
with emphasis on the weak convergence of vector fields up to the boundary of
the domain or compact Riemannian manifold, which requires the divergence
and curl of the vector fields to be bounded in $L^2$.
One of our key observations is that,
for the ``usual'' div-curl lemmas -- with the exception of certain end-point cases,
{\it e.g.}, Theorem \ref{thm_ generalised critical case, div curl lemma},
the specific {\em geometry}
of Euclidean spaces or manifolds plays no essential role.
Based on this observation, we have formulated and established a general compensated compactness theorem
through bounded linear operators on Banach spaces in \cite{chenli}. Roughly speaking, it may be stated as follows:
If two bounded linear operators $S$ and $T$ between Banach spaces satisfy two conditions:
One is algebraic ($S$ and $T$ are {\it orthogonal} to each other), and the other is analytic
($\st$ determines  {\it nearly everything}),
then a result in the spirit of Lemma \ref{proposition: murat-tartar} holds,
with {\it div} and {\it curl} replaced by $S$ and $T$, respectively.

We now discuss the functional-analytic compensated compactness theorem
in Banach spaces, as well as its geometric implications.
For some background on functional analysis, we refer to \cite{f}.
	
Let us first explain some notations:
In the sequel, $\hil$ is a Hilbert space over the field $\mathbb{K}=\real \text{ or } \mathbb{C}$
so that $\hil=\hil^\ast$, and $Y,Z$ are two Banach spaces over $\mathbb{K}$.
We use $\hil^\ast$ and $Y^*, Z^*, \ldots$, to denote the dual Hilbert and Banach spaces, respectively.
In what follows, we consider the bounded linear operators:
$$
S:\hil\to Y, \quad T:\hil\to Z.
$$
For their adjoint operators, we write
$$
\sdag:Y^*\to\hil \quad \tdag: Z^* \to \hil.
$$
By $\langle\cdot,\cdot\rangle_Y, \langle\cdot,\cdot\rangle_Z, \ldots$,
we mean the duality pairings on suitable Banach spaces,
and notation $\langle\cdot,\cdot\rangle$ without subscripts
is reserved for the inner product on $\hil$.
Furthermore, for any normed vector spaces $X$, $X_1$, and $X_2$,
we write $\{s^\epsilon\}\subset X$ for a sequence $\{s^\epsilon\}$ in $X$
as a subset, and $X_1 \Subset X_2$ for a compact embedding between
the normed vector spaces.
We use $\|\cdot\|_X$ to denote the norm in $X$,
write $\rightarrow$ for the strong convergence of sequences under the norm,
and write $\rightharpoonup$ for the weak convergence.
$\bar{B}_X:=\{x\in X: \|x\|\leq 1\}$ is the closed unit ball in $X$,
and $B_X:=\{x\in X: \|x\|<1\}$ is the open unit ball.
Moreover, for a linear operator $L: X_1\mapsto X_2$,
its kernel is written as $\ker(L)\subset X_1$,
and its range is $\ran(L)\subset X_2$.
Finally, for $X_1\subset X$ as a vector subspace,
its {\em annihilator} is defined
as $X_1^\perp:=\{f\in X^*:f(x)=0 \text{ for all } x\in X_1\}$.

To proceed, we
define the following linear operators:
\begin{equation*}
\begin{cases}
\st : \hil \mapsto Y\bigoplus Z, \qquad\,\, &(\st) h:=(Sh,Th);\\[1mm]
\stdag: (Y\bigoplus Z)^*\cong Y^*\bigoplus Z^* \mapsto \hil, \qquad &(\stdag)(a,b) := \sdag a + \tdag b
\end{cases}
\end{equation*}
for $h\in\hil$, $a\in Y^*$, and $b\in Z^*$.
The direct sum $Y\bigoplus Z$ is always endowed with the
norm: $\|(y,z)\|_{Y\bigoplus Z}:=\|y\|_Y+\|z\|_Z$.
Also, it is direct to see that $(\st)^\dagger = \stdag$.

Our compensated compactness theorem is formulated in the following:

\begin{theorem}[Theorem \ref{thm_abstract compensated compactness} in \cite{chenli}]\label{thm_abstract compensated compactness}
Let $\hil=\hil^*$ be a Hilbert space over $\mathbb{K}$,
$Y$ and $Z$ be reflexive Banach spaces over $\mathbb{K}$,
and $S:\hil\mapsto Y$ and $T:\hil\mapsto Z$ be bounded linear operators satisfying
\begin{enumerate}
\item[\rm (Op 1)]
Orthogonality{\rm :}
\begin{equation}\label{eqn_almost ortho two}
	S\circ \tdag = 0, \qquad T\circ \sdag =0{\rm ;}
\end{equation}
\item[\rm (Op 2)]
For some Hilbert space $(\htt; \|\cdot\|_{\htt})$ such that $\hil$ embeds compactly into $\htt$,
there exists a constant $C>0$ so that, for all $h\in\hil$,
\begin{equation}\label{estimate new}
\|h\|_{\hil} \leq C\big(\|(Sh,Th)\|_{Y\bigoplus Z}+\|h\|_{{\htt}}\big)
= C\big(\|Sh\|_Y+\|Th\|_Z + \|h\|_{{\htt}}\big).
\end{equation}
\end{enumerate}

Assume that two sequences $\useq, \vseq \subset \hil$ satisfy the following conditions{\rm :}
\begin{enumerate}
\item[\rm (Seq 1)]
$u^\e \rightharpoonup \ub$ and  $v^\e\rightharpoonup \vbb$ in $\hil$ as $\e \to 0${\rm ;}
\item[\rm (Seq 2)]
$\{Su^\e\}$ is pre-compact in $Y$, and $\{Tv^\e\}$ is pre-compact in $Z$.
\end{enumerate}
Then
$$
\langle u^\e, v^\e\rangle \rightarrow \langle\ub, \vbb\rangle \qquad \text{ as } \e\to 0.
$$
\end{theorem}

\begin{proof}[Outline of Proof]
We now sketch the main steps of the proof here.
The interested readers are referred to \cite{chenli} for the details.

\smallskip
{\bf Step 1.} Claim:  {\it $\st: \hil\mapsto\yz$ has finite-dimensional kernel and closed range}.
As $Y$ and $Z$ are reflexive, $\ran(\st)$ is also a reflexive Banach space.
This observation guarantees that all the assumptions (Op 1)--(Op 2) and (Seq 1)--(Seq 2)
remain valid, provided that $\yz$ is replaced by $\ran(\st)$, {\it i.e.}, $S$ and $T$ are surjective.
Thus, once Step 1 has been established,  we can assume that $\st$ is {\em Fredholm}
in the subsequent arguments.
	
Indeed, to show $\dim_{\mathbb{K}}\ker(\st) < \infty$, by the classical Riesz lemma,
it suffices to check that the closed unit ball of $\ker(\st)$ is compact
in the norm topology of $\yz$.
To this end, let $j: \hil \emb \htt$ be the compact embedding in (Op 2).
Then, for any $h \in \hil$ such that $ j(h)\in j[\ker(\st)]\cap \bar{B}_{\htt}$,
the same condition yields
	\begin{equation}
	\|h\|_{\hil} \leq C(\|Sh\|_Y + \|Th\|_Z + \|j(h)\|_{\htt}) \leq C.
	\end{equation}
Therefore, the unit ball of $j[\ker(\st)]$ in $\htt$ is finite-dimensional,
and the same conclusion holds for $\ker(\st)$ as $j$ is an embedding.

To show $\ran(\st) \subset \yz$ as a closed subspace, we take any sequence $\{h^\mu\}\subset\hil$
such that $(\st) h^\mu \rightarrow w$ in the norm topology of $\yz$ and
argue that $w \in \ran(\st)$.
This follows from the following {\em coercivity estimate}: There exists a universal constant $\e_0 > 0$ such that
\begin{equation}\label{eq: coercive}
\|(\st)h\|_{\yz} \geq \e_0 \|j(h)\|_{\htt} \qquad \text{ for all } h \in \hil.
\end{equation}
	
The estimate in \eqref{eq: coercive} is obtained via a contradiction argument,
by taking into account of the finite-dimensionality of $\ker(\st)$ and the $1$-homogeneity
of \eqref{eq: coercive}.
Then we decompose $h^\mu = k^\mu + r^\mu$ for  $k^\mu \in \ker(\st)$ and $r^\mu \in [\ker(\st)]^\perp$.
In view of the inequality:
$$
\|(\st)(h^{\mu_1}-h^{\mu_2})\|_{\yz} \geq \e_0 \|j (r^{\mu_1}-r^{\mu_2})\|_{\htt},
$$
we find that $\{j(r^\mu)\}$ is a Cauchy sequence in $\htt$, which converges to some $j(r)$.
Then it is direct to check that $(\st)r = w$, which leads to the claim in Step $1$.

\smallskip
Notice in passing that we have obtained the following decomposition of $\hil$ along $\st$:
\begin{equation}\label{eq: decomposition of H}
\hil = \ker(\st) \bigoplus \ran (\stdag),
\end{equation}
where $\bigoplus$ is the topological direct sum of the Banach spaces,
with the summands being orthogonal with respect to the inner product on $\hil$.
Moreover, note that only the {\it analytic} assumption (Op 2) on $S$ and $T$
has been used in Step 1.

\smallskip
{\bf Step 2.}
From now on, $\st$ is assumed to be surjective and with finite-dimensional kernel.
In this step, we decompose each of the two sequences $\{v^\e\}$ and $\{w^\e\}$
into three parts: an $S$-free part, a $T$-free part, and a remainder
in the finite-dimensional space $\ker(\st)$.
This is done via  the {\it generalized Laplacian}.

Indeed, we define operator $\sd: \yzstar \mapsto \yz$ as follows:
\begin{equation}
\sd := (\st) \circ (\stdag) = S\sdag \oplus T\tdag.
\end{equation}
Then, thanks to Eq. \eqref{eq: decomposition of H} and  $\ker(\stdag) = [\ran(\st)]^\perp$,
we find that $\sd$ also has  finite-dimensional kernel and closed range and,
as in Step 1, $\sd$ can be assumed to be surjective.
Denote by $\pi_1:\hil = \ker(\st) \bigoplus \ran (\stdag) \mapsto \ker(\st)$
the canonical projection onto the first coordinate,
which is a finite-rank (hence compact) operator.
Then our decomposition of $\{v^\e\}$ and $\{w^\e\}$ are given as follows:
\begin{equation}\label{decompositions of vector fields u,v}
\begin{cases}
u^\e = \pi_1(u^\e) + \sdag a^\e+\tdag b^\e, \qquad
&v^\e=\pi_1(v^\e) + \sdag \antil + \tdag \bntil,\\[1mm]
\ub = \pi_1(\ub) + \sdag a + \tdag b, \qquad
&\vbb = \pi_1(\vbb) + \sdag \tilde{a} +\tdag \tilde{b},
\end{cases}
\end{equation}
for some $a,\tilde{a}, a^\e,\antil \in Y^*$ and $b,\tilde{b}, b^\e, \bntil\in Z^*$.
Applying the {\em algebraic} condition (Op 1) of operators $S$ and $T$,
the inner products become:
\begin{equation}
\begin{cases}
\langle u^\e, v^\e\rangle = \langle \pi_1(u^\e), \pi_1(v^\e)\rangle + \langle \sdag a^\e, \sdag \antil\rangle
  + \langle \tdag b^\e, \tdag\bntil\rangle,\\[1mm]
\langle \ub, \vbb\rangle = \langle \pi_1(\ub), \pi_1(\vbb)\rangle
 + \langle \sdag a, \sdag \tilde{a}\rangle + \langle \tdag b, \tdag\tilde{b}\rangle.
\end{cases}
\end{equation}
Owing to the compactness of $\pi_1$, $\langle \pi_1(u^\e), \pi_1(v^\e)\rangle  \rightarrow \langle \pi_1(\ub), \pi_1(\vbb)\rangle$
as $\epsilon \rightarrow 0$. Therefore, to conclude the theorem, it remains to establish
\begin{equation}\label{key convergence}
\langle \sdag a^\e, \sdag \antil\rangle + \langle \tdag b^\e, \tdag\bntil\rangle
\rightarrow \langle \sdag a, \sdag \tilde{a}\rangle + \langle \tdag b, \tdag\tilde{b}\rangle \qquad \text{ as } \e \rightarrow 0,
\end{equation}
which is the content of the next step.

\smallskip
{\bf Step 3.}
To prove the convergence in \eqref{key convergence}, we start with the following two observations:
\begin{enumerate}
\item[\rm (i)] The left-hand side of \eqref{key convergence} can be expressed in terms of the generalized Laplacian:
\begin{align}\label{aa}
&\langle \sdag a^\e, \sdag \antil\rangle + \langle \tdag b^\e, \tdag\bntil\rangle \nonumber\\
&=\langle S\sdag a^\e,\antil\rangle_{Y} + \langle b^\e, T\tdag \bntil\rangle_{Z}=\big\langle\sd (a^\e,\bntil), (\antil,b^\e)\big\rangle_{\yz};
\end{align}
\item[\rm (ii)] Multiplying $S$ to $u^\e$ and $T$ to $v^\e$ in \eqref{decompositions of vector fields u,v} and invoking (Op 1),
we have
\begin{equation}\label{aaa}
Su^\e = S\sdag a^\e, \qquad Tv^\e = T\tdag \bntil,
\end{equation}
so that
\begin{equation}
\sd  (a^\e,\bntil) = (Su^\e, Tv^\e).
\end{equation}
\end{enumerate}

Now, as $\{Su^\e\}\subset Y$ and $\{Tv^\e\}\subset Z$ are pre-compact by assumption (Seq 2),
it suffices to show the boundedness of $\{(\antil, b^\e)\}$ in the norm topology of $\yzstar$
to reach the conclusion.
Furthermore, in view of the specific form of the expression involved in \eqref{key convergence},
it is enough to exhibit one particular representative $(\antil, b^\e)$ in
the co-set $\sd^{-1}\{(Sv^\e, Tu^\e)\}$
such that $\|(\antil, b^\e)\|_{\yzstar}\leq C$,
where $C>0$ is independent of $\e$.
As $\{(Sv^\e, Tu^\e)\}$ is uniformly bounded in the norm topology of $\yz$,
owing to the weak convergence of $\{v^\e\}$ and $\{w^\e\}$ assumed in (Seq 1),
the desired result follows from a standard result in functional analysis,
which is Claim $\clubsuit$ in the proof of Theorem 3.1 in \cite{chenli}.
This completes the proof.
\end{proof}

\smallskip
With the benefit of hindsight, let us now explain the motivation for Theorem \ref{thm_abstract compensated compactness}
and its relations with the earlier versions of the div-curl lemmas.
Consider a $3$-dimensional oriented closed manifold $M$ (differentiable, or of weaker Sobolev regularity, not necessarily Riemannian).
We denote by $\Omega^q(M)$ the space of differential $q$-forms on $M$,
by $\ast: \Omega^q(M)\mapsto \Omega^{\dim(M)-q}(M)$ the Hodge-star,
by $d: \Omega^q(M) \mapsto \Omega^{q+1}(M)$ the exterior differential,
and by $\sharp$ the tonic operator, {\it i.e.}, the canonical isomorphism between
the co-tangent bundle $T^*M$ and the tangent bundle $TM$ by raising indices
in the coefficients.
It is well-known that {\em div}, {\em grad}, and {\em curl}
can be defined intrinsically via the commutative diagram:
\begin{center}
$\begin{array}[c]{ccccccc}
\Omega^0(M)&\stackrel{d}{\longrightarrow}&\Omega^1(M)&\stackrel{d}{\longrightarrow}&\Omega^2(M)&\stackrel{d}{\longrightarrow}&\Omega^3(M)\\
\downarrow\scriptstyle{{\rm Id}}&&\downarrow\scriptstyle{\sharp}&&\downarrow\scriptstyle{\sharp\circ\ast}&&\downarrow\scriptstyle{\ast}\\
C^\infty(M)&\stackrel{{\rm grad}}{\longrightarrow}&\Gamma(TM)&\stackrel{\curl}{\longrightarrow}&\Gamma(TM)&\stackrel{\di}{\longrightarrow}&C^\infty(M)
\end{array}$
\end{center}
In particular, the Riemannian metric on $M$ plays no role at all.
The ``orthogonality'' of {\em div} and {\em curl} in the sense of (Op 1)
in Theorem \ref{thm_abstract compensated compactness},
which follows from the cohomological chain condition $d\circ d = 0$,
$\delta \circ \delta =0$, is a purely algebraic relation.
Therefore, it is not surprising that a compensated compactness theorem
with greater generality and abstractness is available.

Moreover, the Hodge decomposition approach to the div-curl lemma
initiated by Robbin-Rogers-Temple in \cite{RRT87} makes use of the Laplacian $\Delta$ on flat $\R^3$.
If we take $S=\di$, $T=\curl$, $\hil = L^2(\R^3;\R^3)$, $Y=H^{-1}(\R^3; \R)$, and $\htt=Z=H^{-1}(\R^3; \R^3)$
with suitable localizations if necessary,
the classical div-curl lemma (Lemma \ref{proposition: murat-tartar}) is immediately recovered.
Our generalized Laplacian $\sd$ extends the flat Laplacian and,
more generally, the Laplace-Beltrami operator $\Delta$ on manifolds,
in view of Eqs. \eqref{eq: hodge in R3}--\eqref{eq: hodge for any mfd}.
The Fredholmness of $\Delta$ follows from the Hodge decomposition theorem, {\it cf.} \S 6 in \cite{War71}.

Before our subsequent development, we remark that
the assumption of the reflexivity of $Y$ and $Z$ is crucial, since several counterexamples have been
constructed for non-reflexive $Y$ and $Z$ (see \cite{CDM} and Remark 3.2 in \cite{chenli}).
	
Now we discuss a geometric consequence of Theorem \ref{thm_abstract compensated compactness}.
Using the expression for $\sd=\Delta$ on Riemannian manifolds in terms of $S=d$ and $T=\delta$
as in Eq. \eqref{eq: hodge for any mfd}, we have

\begin{theorem}[Geometrically intrinsic div-curl lemma A, Theorem 3.3 in \cite{chenli}]\label{thm: geometric div-curl}
Let $(M,g)$ be an $n$-dimensional Riemannian manifold.
Let $\{\omega^\epsilon\}, \{\tau^\epsilon\} \subset L^2_{\rm loc}(M;\bigwedge^qT^*M)$ be two families
of differential $q$-forms such that
\begin{enumerate}
\item[\rm (i)]
$\omega^\epsilon \rightharpoonup \overline{\omega}$
and $\tau^\epsilon \rightharpoonup \overline{\tau}$ weakly in $L^2_{\rm loc}(M; \bigwedge^{q}T^*M)${\rm ;}
\item[\rm (ii)] There are compact subsets of the corresponding Sobolev spaces, $K_d$ and $K_\delta$, such that
\begin{equation*}
\begin{cases}
\{d\omega^\epsilon\}\subset K_d \Subset H^{-1}_{\rm loc}(M;\bigwedge^{q+1}T^*M),\\[1.5mm]
\{\delta\tau^\epsilon\}\subset K_\delta \Subset H^{-1}_{\rm loc}(M;\bigwedge^{q-1}T^*M).
\end{cases}
\end{equation*}
\end{enumerate}
Then $\langle\omega^\epsilon, \tau^\epsilon\rangle$ converges to $\langle\overline{\omega}, \overline{\tau}\rangle$
in $\mathcal{D}'(M)$, that is,
\begin{equation*}
\int_M \langle\omega^\epsilon, \tau^\epsilon\rangle\psi\, {\rm d}V_g
\longrightarrow \int_M \langle\overline{\omega}, \overline{\tau}\rangle \psi\, {\rm d}V_g
\qquad\,\, \mbox{for any $\psi \in C^\infty_{c}(M)$}.
\end{equation*}
\end{theorem}
	
Since the conclusion for the weak continuity in Theorem \ref{thm: geometric div-curl} is in the distributional sense,
we may assume $M$ to be oriented and closed without loss of generality in the proof.
Here and in the sequel, $W^{k,p}(M; \qtens)$ denotes the Sobolev space of differential $q$-forms
with $W^{k,p}$--regularity.
Then $\Delta: \Omega^q(M)\mapsto \Omega^q(M)$, as well as $\Delta: W^{k,p}(M;\qtens) \mapsto W^{k-2,p}(M;\qtens)$,
for $0\leq q\leq n=\dim(M)$, is elliptic, which is crucial for the verification of condition (Op 2)
in Theorem \ref{thm_abstract compensated compactness}.
As is well-known, the analogous operator on semi-Riemannian manifolds
is not elliptic in general, so that Theorem \ref{thm: geometric div-curl} may not be extended
directly to the semi-Riemannian settings.

Next, we state an endpoint case of the above theorem, for which the first-order differential
constraints are prescribed in {\em non-reflexive} Banach spaces $W^{-1,1}$,
in contrast to condition (Seq 2) in Theorem \ref{thm_abstract compensated compactness}.
The underlying argument for the proof essentially follows from that in Conti-Dolzmann-M\"{u}ller (\cite{CDM}),
which employs a Lipschitz truncation argument and the pre-compactness theorems for $L^1$ ({\it e.g.},
Chacon's biting lemma, the Dunford-Pettis theorem, etc.) to reduce to the reflexive case.

\begin{theorem}[Geometrically intrinsic div-curl lemma B, Theorem 6.1 in \cite{chenli}]\label{thm_ generalised critical case, div curl lemma}
Let $(M,g)$ be an $n$-dimensional manifold.
Let $\{\omega^\epsilon\}\subset L^2_{\rm loc}(M;\qtens)$
and $ \{\tau^\epsilon\}\subset L^2_{\rm loc}(M;\qtens)$ be two families of  differential $q$-forms. Suppose that
\begin{enumerate}
\item[\rm (i)]
$\omega^\epsilon \rightharpoonup \overline{\omega}$ and $\tau^\epsilon \rightharpoonup \overline{\tau}$ weakly
in $L^2_{\rm loc}(M;\qtens)$ as $\e\to 0${\rm ;}
\item[\rm (ii)] There are compact subsets of the corresponding Sobolev spaces, $K_d$ and $K_\delta$, such that
\begin{equation*}
\begin{cases}
\{d\omega^\epsilon\}\subset K_d \Subset W^{-1,1}_{\rm loc}(M;\bigwedge^{q+1}T^*M),\\[1.5mm]
\{\delta\tau^\epsilon\}\subset K_\delta \Subset W^{-1,1}_{\rm loc}(M;\bigwedge^{q-1}T^*M);
\end{cases}
\end{equation*}
\item[\rm (iii)] $\{\langle\omega^\epsilon,\tau^\epsilon\rangle\}$ is  equi-integrable.
\end{enumerate}
Then $\langle\omega^\epsilon, \tau^\epsilon\rangle$ converges to $\langle\overline{\omega}, \overline{\tau}\rangle$
in $\mathcal{D}'(M)$, that is,
\begin{equation*}
\int_M \langle\omega^\epsilon, \tau^\epsilon\rangle\psi\, {\rm d}V_g
\longrightarrow \int_M \langle\overline{\omega}, \overline{\tau}\rangle \psi\, {\rm d}V_g
\qquad\,\, \mbox{for any $\psi \in C^\infty_{c}(M)$}.
\end{equation*}
\end{theorem}

To conclude this section, we remark that the preceding intrinsic div-curl lemmas
(Theorems \ref{thm: geometric div-curl}--\ref{thm_ generalised critical case, div curl lemma})
can be extended to the case of general H\"{o}lder exponents,
namely that $\{\omega^\e\}\subset L^r_\loc(M; \qtens)$ and $\{\tau^\e\}\subset L^s_\loc(M;\qtens)$
with $\frac{1}{r}+ \frac{1}{s}=1$.
Since $L^r$ is not a Hilbert space unless $r=2$,
such generalizations cannot be directly deduced from  Theorem \ref{thm_abstract compensated compactness}.
Nevertheless, they follow from similar arguments, with slight modifications in light of Eq. \eqref{eq: hodge for any mfd}.
We refer to Theorems 3.7, Theorem 6.2, and Appendix in \cite{chenli} for the details;
also see \S 5 in
\cite{KY13}.

\section[\small Isometric Immersions \& the GCR Equations]{Isometric Immersions
of Riemannian Manifolds \\and the Gauss-Codazzi-Ricci (GCR) Equations}

In this section, we briefly discuss the geometric preliminaries
for the isometric immersions of Riemannian manifolds.
We restrain ourselves to the constructions directly related to our subsequent development.
We refer to the classical texts \cite{DoC92, Eisenhart, Spi79} for more detailed treatments on differential geometry,
to Han-Hong \cite{HanHon06}, as well as the classical papers by Nash \cite{Nas54, Nas56}, for isometric immersions.

From now on, let $M$ be an $n$-dimensional Riemannian manifold, and let $g$ be a Riemannian metric on $M$.
Motivated by the applications in nonlinear elasticity ({\it cf.} \cite{Ball, CSW10-CMP, Cia08, Mar05}),
we consider the metrics of weaker regularity: $g\in W^{1,p}_\loc$.
A $W^{2,p}_\loc$ map $f: (M,g) \mapsto (\R^{n+k}, g_0)$ for the Euclidean metric $g_0$
is an {\em immersion} if the differential $df_P$ is injective for each $P\in M$,
and it is an {\em embedding} if $f$ itself is also injective.
Moreover, $f$ is {\em isometric} if
\begin{equation}
df \otimes df = g,
\end{equation}
that is, for any $P \in M$ and vector fields $X,Y$ on $M$,
\begin{equation}\label{eq: def of isometric}
df_P(X) \cdot df_P(Y) = g_P(X, Y),
\end{equation}
where $g_P$ denotes the metric evaluated at $P$,
and $\cdot$ is the Euclidean dot product on $\rnk$.
Notice that Eq. \eqref{eq: def of isometric} makes sense in the distributional sense
when $p^*=\frac{np}{n-p}\geq 2$, and that $g$ has a continuous representative when $p>n$,
in view of the Sobolev embeddings $W^{2,p}(\R^n)\emb W^{1,q}(\R^n)$
for $1 \leq q \leq p^*$ and $W^{1,p} (\R^n) \emb C^0(\R^n)$ for $p>n$.

In differential geometry,
the description of an isometric immersion is equivalent to the determination
of how $\rnk$ -- viewed as its own tangent spaces -- can be split into
the immersion-independent and immersion-dependent geometry of $M$.
More precisely,
for each point $P\in M$, we have the vector space direct sum
\begin{equation}
\rnk = T_PM \bigoplus T_PM^\perp,
\end{equation}
where $T_PM$ is the {\em tangent space} of $M$ at $P$,
and $T_PM^\perp$ is its complement in $\rnk$,
interpreted as the {\em normal space}.

To study the isometric immersions,
two approaches have been employed from the PDE point of view.
One is to deal directly with Eq. \eqref{eq: def of isometric},
which is a first-order, nonlinear, generally under-determined PDE;
the other is to derive a PDE system by taking two more derivatives
and solves for the compatibility conditions.
The former approach has been employed by Nash \cite{Nas54, Nas56}
to establish the existence of $C^1$ and $C^k$ isometric embeddings
for large enough co-dimensions; also see G\"{u}nther \cite{Gunther89}
for a simplification.
For the latter approach,
the compatibility conditions read schematically as follows:
\begin{align*}
\quad 0 &= \text{Curvature of } \rnk \\
&\Longleftrightarrow \,  \begin{cases}
\text{Curvature in (tangential, tangential) direction}=0,\\
\text{Curvature in (tangential, normal) direction}=0, \\
\text{Curvature in (normal, normal) direction}=0.
\end{cases}\qquad(\spadesuit)
\end{align*}

To continue, we use Latin letters $X,Y,Z,W,\ldots$ to denote the tangential vector fields
in $\Gamma(TM)$, write Greek letters $\xi, \eta, \zeta, \ldots$ for the  normal
vector fields in $\Gamma(TM^\perp)$, and identify vector fields with
first-order differential operators.
Then $XY$ and $X\xi$ are vector fields in $\Gamma(TM)$
and $\Gamma(TM^\perp)$, respectively.
Let us consider the well-known geometric quantities:

\begin{itemize}
\item
{\bf Immersion-independent quantities.} Taking one derivative in $g$ leads to
the Levi-Civita connection (or covariant derivative) $\na: \Gamma(TM)\times \Gamma(TM) \mapsto \Gamma(TM)$,
and taking one further derivative gives us the Riemann curvature tensor
$R:  \Gamma(TM)\times\Gamma(TM)\times\Gamma(TM)\times\Gamma(TM) \mapsto \R$.

\item
{\bf Immersion-dependent quantities.} For given $\na, R$ as above, consider the isometric immersion $f:(M,g)\emb (\rnk, g_0)$.
We define the {\em second fundamental form} $B: \Gamma(TM) \times \Gamma(TM) \mapsto \Gamma(TM^\perp)$ by
\begin{equation}
B(X,Y):= XY-\na_X Y,
\end{equation}
and the {\em normal connection} $\na^\perp:\Gamma(TM) \times \Gamma(TM^\perp) \mapsto \Gamma(TM^\perp)$ by
\begin{equation}
\na^\perp_X\xi := \text{ projection of } X\xi \text{ (at each point) onto } TM^\perp.
\end{equation}
\end{itemize}
Then the right-hand sides of the schematic equations in $(\spadesuit)$ can be expressed
via the quantities $(g, \na, R, B, \na^\perp)$, resulting in the {\em Gauss},
{\em Codazzi}, and {\em Ricci} equations in  \eqref{eqn_gauss},  \eqref{eqn_codazzi}, and
 \eqref{eqn_ricci}, respectively, below; see also Theorems 2.1--2.2 in \cite{chenli}
 and \S 6 in \cite{DoC92}.

\begin{proposition}\label{propn: GCR equations}
Suppose that $f\in W^{2,p}_\loc(M,g; \rnk, g_0)$ for $p>n$ is an isometric immersion.
Then the following compatibility conditions are satisfied in $\mathcal{D}'${\rm :}
\begin{eqnarray}
&&\langle B(X,W), B(Y,Z)\rangle - \langle B(Y,W), B(X,Z) \rangle =R(X,Y,Z,W), \label{eqn_gauss}\\[3mm]
&&XB(Y,Z,\eta) - YB(X,Z,\eta)\nonumber\\
&&= ([X,Y],Z,\eta)	- B(X,\na_YZ,\eta) - B(X,Z,\nap_Y\eta)\nonumber\\
&&\quad  + B(Y,\na_XZ,\eta) + B(Y,Z,\nap_X\eta), \label{eqn_codazzi}\\[3mm]
&&X\langle \nap_Y \xi,\eta\rangle - Y \langle \nap_X\xi,\eta \rangle\nonumber\\
&&= \langle \nap_{[X,Y]}\xi,\eta \rangle - \langle \nap_X \xi, \nap_Y \eta\rangle
+ \langle \nap_Y\xi, \nap_X \eta\rangle\nonumber\\
&&\quad
+B(X\xi - \nabla^\perp_X \xi, Y, \eta) - B(X\eta - \nabla^\perp_X \eta, Y, \xi),
\label{eqn_ricci}
\end{eqnarray}
where $X,Y,Z,W\in \Gamma(TM)$, $\eta,\xi \in \Gamma(TM^\perp)$, and $[X,Y]=XY-YX$ is the Lie bracket.
Here and in the sequel, we have used $\langle\cdot,\cdot\rangle$
to denote all the inner products induced by  metrics, and $B(Y,Z,\eta):=\langle B(Y,Z),\eta\rangle$.
\end{proposition}

From the PDE perspectives, we view the immersion-dependent quantities $(B,\na^\perp)$
as to be solved, and the immersion-independent quantities $(g, \na, R)$ as being fixed.
Indeed, in the isometric immersion problem, metric $g$ is prescribed,
so are all the immersion-independent quantities; thus, the immersion-dependent geometry determines
the whole of the isometric immersion.
Therefore, in the sequel, {\em the GCR equations are always considered as a first-order
nonlinear PDE system for $(B,\na^\perp)$}.

Proposition \ref{propn: GCR equations} says that the GCR equations form
a necessary condition for the existence of isometric immersions.
The converse is known as the ``{\em realization problem}'' in elasticity:
Given
$(B, \na^\perp)$  satisfying the GCR equations,
construct an isometric immersion ({\it i.e.}, design an elastic body) whose immersion-dependent geometry
is prescribed by $(B, \na^\perp)$.
This problem for both $C^\infty$ and $W^{1,p}_\loc$ metrics has been answered
in the affirmative, globally on simply-connected manifolds;
see Tenenblat \cite{Ten71} for the former, and Mardare \cite{Mar05, Mar07} and Szopos \cite{Szo08}
for the latter.
In \cite{chenli}, we adapt the geometric arguments in \cite{Ten71} to re-prove
the realization theorem in $W^{1,p}_\loc$ regularity, which simplifies the proofs
in \cite{Mar05, Mar07, Szo08}. Moreover, this method sheds light on the weak rigidity
problem of isometric immersions, which is the main content of \S 4.

Finally, we briefly sketch the main tool in \cite{Ten71} -- {\em the Cartan formalism} -- which
serves as a bridge between the geometric problem of isometric immersions and the PDEs (GCR equations).
In full generality, consider a vector bundle $E$ over $M$ of fibre $\R^k$, trivialized
on a local chart $U \subset M$, {\it i.e.}, $E|_U \cong U \times \R^k$ as a diffeomorphism.
Let $\{\partial_i\} \subset \Gamma(TU)$ be an orthonormal frame,
and let $\{\omega^i\} \subset \Omega^1(U)$ be its dual (co-frame).
Then we choose  $\{\eta_{n+1}, \ldots, \eta_{n+k}\} \subset \Gamma(E)$ as an orthonormal
basis for fibre $\R^k$, and set
\begin{eqnarray}
&&\omega^i_j (\partial_k):= \langle \na_{\partial_k}\partial_j, \partial_i \rangle,
\label{eqn_def of connection form 1}\\[1mm]
&&\omega^i_\alpha(\partial_j)=-\omega^\alpha_i(\partial_j):=\langle B(\partial_i, \partial_j),\eta_\alpha\rangle,
\label{eqn_def of connection form 2}\\[1mm]
&&\omega^\alpha_\beta (\partial_j):= \langle\na^E_{\partial_j}\eta_\alpha, \eta_\beta \rangle,
\label{eqn_def of connection form 3}
\end{eqnarray}
where $\na^E$ is the bundle connection: $\na^E=\na^\perp$ for $E=TM^\perp=$ the {\em normal bundle}.
All these  constructions make sense in distributions
for $g\in W^{1,p}_\loc$.
Moreover, here and in the sequel, the following index convention is adopted:
\begin{equation*}
1\leq i,j \leq n;\qquad 1\leq a,b,c,d,e \leq n+k; \qquad n+1 \leq \alpha,\beta,\gamma \leq n+k.
\end{equation*}

In this setting, the GCR equations on bundle $E$ are equivalent to the following two systems,
known as the {\em first} and {\em second structural equations of the Cartan formalism}
({\it cf.} \cite{Ten71, Spi79, sternberg}):
\begin{eqnarray}
&& d\omega^i = \sum_j \omega^j \wedge \omega^i_j, \label{eqn_first structure eqn}\\
&& d\omega^a_b = - \sum_c \omega_b^c \wedge \omega^a_c  \label{eqn_second structure eqn}
\end{eqnarray}
for each $i, a, b$.
These equations can be represented compactly as first-order nonlinear Lie algebra-valued PDEs.
Denoting by $\mathfrak{so}(n+k)$ the Lie algebra of antisymmetric $(n+k) \times (n+k)$ matrices,
we can write \eqref{eqn_first structure eqn}--\eqref{eqn_second structure eqn} as
\begin{equation}\label{eq: structure eqns, invariant form}
dw=w\wedge W, \qquad dW+ W\wedge W = 0,
\end{equation}
where $W=\{w^a_b\} \in \Omega^1(U; \mathfrak{so}(n+k))$ which is known as the {\em connection one-forms},
$w=(\omega^1,\ldots,\omega^n, 0, \ldots, 0)^\top \in \Omega^1(U; \rnk)$,
and $\wedge$ operates by the wedge product on the factor of differential forms
and matrix multiplication on
the factor of the matrix Lie algebra, with respect to the factorization:
\begin{equation}\label{eq:lie algebra factorization}
\Omega^1(U; \mathfrak{so}(n+k)) \cong \Gamma(\bigwedge^1 T^*U \bigotimes \mathfrak{so}(n+k)).
\end{equation}
In other words, the structural equations recast in \eqref{eq: structure eqns, invariant form}
are also intrinsic, {\it i.e.},
independent of the choice of local moving frames/coordinates  $\{\p_i\}$ and $\{\eta_{\alpha}\}$.

\section[\small Weak Rigidity of Isometric Immersions]{Weak Rigidity of Isometric Immersions}

Finally, we discuss the weak rigidity of isometric immersions with
weaker regularity in $W^{2,p}_\loc$ as in the previous sections.

The {\em rigidity problem} of isometric immersions concerns the following:
{\it If $\{f^\e\}$ is a sequence of isometric immersions of a manifold $M$
into $(\rnk, g_0)$, which converges to a map $f: M \mapsto \rnk$ in a certain topology,
is $f$ still an isometric immersion?}
This problem has a history of celebrated results.
Nash in \cite{Nas54} showed that the $C^1$ isometric immersions are not rigid.
In particular, any $C^\infty$ short ({\it i.e.}, distance-shrinking) immersion is $C^0$-close to an $C^1$ isometric
immersion; see also \cite{torus} for a recent computer visualisation.
In the same sense, Borisov in \cite{borisov} proved that $C^{1,\alpha}$ isometric immersions
are not rigid for $\alpha>0$ below a certain value, and this value has been
improved in \cite{conti-delellis}.
On the other hand, the $C^{1,\alpha}$ isometric immersions for large enough $\alpha$ are classically known
to be rigid; {\it cf.} \cite{Nas56} and the references therein.
More recently, deep connections have been established between the transition phenomenon
from the non-rigidity to rigidity of the $C^{1,\alpha}$ isometric immersions and Onsager's
conjecture (concerning the dissipative  weak solutions to the Euler equations in fluid dynamics).
We refer the readers to \cite{delellis} and the references cited therein  for such developments.

Our focus is on the {\em weak rigidity problem} motivated by applications.
In this case, for the sequence of isometric immersions $\{f^\e\}$ that is weakly and locally convergent in $W^{2,p}$ for $p>n=\dim(M)$,
we ask if the weak limit $\bar{f}$ is still a $W^{2,p}_\loc$ isometric immersion.
Indeed, we answer the question in the affirmative, thanks to the {\em locally uniform $L^p$ bounds on the immersion-dependent geometry}.
This is in the spirit of the works by Langer \cite{Langer} and the recent generalization by Breuning \cite{breuning}.

Our result can be formulated as follows:

\begin{theorem}[Corollary 5.2 in \cite{chenli}]\label{thm: weak rigidity}
Let $M$ be an $n$-dimensional simply-connected Riemannian manifold
with $W^{1,p}_{\rm loc}$ metric $g$ for $p>n$.
Suppose that $\{f^\epsilon\}$ is a family of isometric immersions of $M$ into $\rnk$
with Euclidean metric, uniformly bounded in $W^{2,p}_\loc(M;\rnk)$,
whose second fundamental forms and normal connections are $\{B^\epsilon\}$ and $\{\nep\}$, respectively.
Then, after passing to subsequences, $\{f^\epsilon\}$ converges to $\bar{f}$
weakly in $W_{{\rm loc}}^{2,p}$ which is still an isometric
immersion $\bar{f}: (M,g) \rightarrow \rnk$.
Moreover, the corresponding second fundamental form $\bar{B}$ is a weak limit of $\{B^\epsilon\}$,
and the corresponding normal connection $\overline{\na^\perp}$
is a weak limit of $\{\nep\}$,
both taken in the weak topology in $L^p_{\rm loc}$.
\end{theorem}

\begin{proof}[Outline of Proof]
We sketch the proof in three steps. For the details, we refer to \S 5 (Step $1$),  \S 4.2--\S 4.3 (Step $2$),
and \S 4.4 (Step $3$) in \cite{chenli}.

\smallskip
{\bf Step 1.}
We show the equivalence between the existence of $W^{2,p}_\loc$ isometric immersions and
the existence of $W^{1,p}_\loc$ solutions of the GCR equations
in distributions (Proposition \ref{propn: GCR equations}).
Then the weak rigidity of isometric immersions is translated to the weak rigidity of the GCR equations.
	
Indeed, at the end of \S 3, it is remarked that the GCR equations are equivalent
to the structural equations \eqref{eq: structure eqns, invariant form} of the Cartan formalism.
Hence, by Proposition \ref{propn: GCR equations}, Eq. \eqref{eq: structure eqns, invariant form} is a necessary
condition for the existence of isometric immersions.
Conversely, we follow the arguments in \cite{Ten71} to transform Eq. \eqref{eq: structure eqns, invariant form} into
first-order nonlinear PDEs on Lie groups.
More precisely, the isometric immersion $f$ satisfies the following equations (formulated as initial value problems)
for $A\in W^{1,p}_{\loc}(U;O(n+k))$, where $O(n+k)$ is the group of $(n+k) \times (n+k)$ symmetric matrices:
\begin{equation}\label{eq: pfaff and poincare}
W=dA \cdot A^\top, \qquad df= w\cdot A.
\end{equation}

The above two equations are known as the {\em Pfaff} and the {\em Poincar\'{e}} systems.
In the smooth case, they can be solved by the Frobenius theorem,
by checking that the solution distribution is involutive.
For the weak regularity case, we apply the
theorems due to Mardare \cite{Mar05, Mar07}
for the existence of solutions to Eq. \eqref{eq: pfaff and poincare}.
Then $df\in W^{1,p}_\loc \emb C^0_{\loc}$ for $n>p$,
and it is non-degenerate and distance-preserving, thanks to the Poincar\'{e} system
and the definition of $W$. This implies that $f$ is indeed an isometric immersion.

\smallskip
{\bf Step 2.}
The GCR equations in Proposition \ref{propn: GCR equations}  are reformulated to manifest the {\em div-curl structures},
which admits the application of the intrinsic div-curl lemma, Theorem \ref{thm: geometric div-curl}.

For this purpose, let us fix the tangential vector field $Z$ and normal vector fields $(\xi, \eta)$,
and define the $2$-tensor fields $V^{(B)}_{Z,\eta}, V^{(\na^\perp)}_{\xi,\eta}: \Gamma(TM)\times\Gamma(TM) \mapsto \Gamma(TM)$
and $1$-forms $\Omega^{(B)}_{Z,\eta}, \Omega^{(\na^\perp)}_{\xi,\eta}$ as follows:
\begin{eqnarray*}
&&\vb_{Z,\eta}(X,Y):= B(X,Z,\eta)Y-B(Y,Z,\eta)X, \\
&&\vnab_{\xi,\eta}(X,Y):=\langle \nap_Y \xi, \eta\rangle X - \langle \nap_X \xi, \eta\rangle Y,\\
&&\ob_{Z,\eta}:= -B(\bullet, Z, \eta),\\
&&\onab_{\xi,\eta}:= \langle \nap_{\bullet} \xi, \eta\rangle.
\end{eqnarray*}
For simplicity, we often drop the indices in both $\Omega$ and $V$ from now on.

To wit, these $\Omega$'s are nothing but the contractions of $(B,\na^\perp)$, and the $V$'s are obtained
by applying $\Omega$ to the $2$-Grassmannian ({\it i.e.}, the space of $2$-planes) in $TM$
and polarized in the anti-symmetric fashion.
Recall that the divergence can be defined intrinsically on manifolds
by $\di X:= \ast (\mathcal{L}_X dV_g)$, where $\mathcal{L}$ denotes the Lie derivative,
and the following well-known identities hold on manifolds:
\begin{equation*}
\begin{cases}
\mathcal{L}_X = d \circ \iota_X + \iota_X \circ d \qquad \text{ for } X\in \Gamma(TM),\\
d\alpha(X,Y)=X\alpha(Y) - Y\alpha(X) - \alpha([X,Y])\quad \text{ for } \alpha \in \Omega^1(M), X,Y\in \Gamma(TM).
\end{cases}
\end{equation*}
Thus, the divergence of $V$'s and the {\it generalized curl} ({\it i.e.}, $d$) of $\Omega$'s can be expressed as
\begin{eqnarray}
{\rm div}\big(\vb_{Z,\eta}(X,Y)\big)
  &=&YB(X,Z,\eta) - XB(Y,Z,\eta) \nonumber\\
  && +B(X,Z,\eta)\,\di Y - B(Y,Z,\eta)\,\di X,\qquad \label{eqn_div VB}\\[1mm]
{\rm div}\big(\vnab_{\xi,\eta}(X,Y)\big)
 &=&-Y\langle \nap_X\xi,\eta\rangle +X\langle \nap_Y\xi,\eta\rangle \nonumber\\
 &&+\langle \na^\perp_Y\xi,\eta\rangle\,\di X - \langle \na^\perp_X\xi,\eta\rangle\,\di Y,\label{eqn_div V NABLA}\\[1mm]
d\big(\ob_{Z,\eta}\big)(X,Y) &=& YB(X,Z,\eta) - XB(Y,Z,\eta) + B([X,Y],Z,\eta),\qquad\quad \label{eqn_curl Omega B}\\
d \big( \onab_{\xi,\eta}\big)(X,Y)& =&  -Y\langle \nap_X\xi,\eta\rangle
+X\langle \nap_Y\xi,\eta\rangle - \langle \nap_{[X,Y]}\xi,\eta\rangle,\label{eqn_curl Omega NABLA}
\end{eqnarray}
where the terms $B(X,Z,\eta)\,\di Y$, $B(Y,Z,\eta)\,\di X$, $\langle \na^\perp_Y\xi,\eta\rangle\,\di X$,
$\langle \na^\perp_X\xi,\eta\rangle\,\di Y$, $B([X,Y],Z,\eta)$, and $\langle \nap_{[X,Y]}\xi,\eta\rangle$
are linear in $(B,\na^\perp)$, while the other terms on the right-hand sides of the above four equations
involve first-order derivatives of $(B,\na^\perp)$.
Moreover, for further development, it is crucial to observe that
\begin{equation}\label{eq: div V equals d Omega}
\begin{cases}
{\rm div}\big(\vb_{Z,\eta}(X,Y)\big) = d\big(\ob_{Z,\eta}\big)(X,Y) + [\text{linear terms}],\\[1mm]
{\rm div}\big(\vnab_{\xi,\eta}(X,Y)\big)  =  d \big( \onab_{\xi,\eta}\big)(X,Y)  + [\text{linear terms}].
\end{cases}
\end{equation}

Next, using the tensor fields $V$ and $\Omega$ introduced above,
we can reformulate the GCR system as the following equations
with emphasis on the pairings of $V$'s and $\Omega$'s:
\begin{eqnarray}
&&\sum_\eta \langle\vb_{Z,\eta}(X,Y), \ob_{W,\eta} \rangle= R(X,Y,Z,W),
\label{eqn_Gauss Eqn in second computation}\\
&&d(\ob_{Z,\eta})(X,Y) + \sum_\beta \langle \vnab_{\eta,\beta}(X,Y),\ob_{Z,\beta}\rangle
\nonumber\\
&&\qquad\quad
+ B(Y, \na_XZ, \eta) - B(X, \na_YZ,\eta)=0, \qquad
\label{eqn_Codazzi Eqn in second computation}\\[2mm]
&& d(\onab_{\xi,\eta})(X,Y) + \sum_{\beta} \langle \vnab_{\eta,\beta}(X,Y), \onab_{\xi,\beta} \rangle
= \sum_Z\langle\vb_{Z,\xi}(X,Y),\ob_{Z,\eta}\rangle, \qquad\,\,\,
\label{eqn_Ricci Eqn in second computation}
\end{eqnarray}
where all the summations are at most countable and locally finite.

Therefore, we have transformed the GCR equations in Proposition \ref{propn: GCR equations}
into Eqs. \eqref{eqn_Gauss Eqn in second computation}--\eqref{eqn_Ricci Eqn in second computation},
expressed in terms of the tensor fields $V$ and $\Omega$.
Furthermore, the divergence of $V$ roughly equals to the generalized curl of the corresponding $\Omega$,
which involves the derivatives of solutions $(B,\na^\perp)$ up to the first order.

\smallskip
{\bf Step 3.} Now we are at the stage of applying the geometrically intrinsic div-curl lemma (Theorem \ref{thm: geometric div-curl})
to conclude the weak rigidity of isometric immersions.
Let $\{B^\e, \na^{\perp,\e}\}$ be the second fundamental forms and normal connections
associated to the sequence of isometric immersions $\{f^\e\}$.
As $\{f^\e\}$ is uniformly bounded in $W^{2,p}_\loc$,
the tensor fields $\{V^{(B^\e)}, \Omega^{(B^\epsilon)}, V^{(\na^{\perp,\e})}, \Omega^{(\na^{\perp,\e})}\}$
are uniformly bounded in $L^p_\loc$,
so that they are pre-compact in the weak topology.
In view of Eqs. \eqref{eqn_Codazzi Eqn in second computation}--\eqref{eqn_Ricci Eqn in second computation}
and \eqref{eq: div V equals d Omega},
the Cauchy-Schwarz inequality immediately yields that $\{\di V^{(B^\e)}, d\Omega^{(B^\epsilon)}, \di V^{(\na^{\perp,\e})}, d\Omega^{(\na^{\perp,\e})}\}$
are uniformly bounded in $L^{p/2}_\loc$,
which compactly embeds into $W^{-1, p'}_\loc$ for some $1<p'<2$.
On the other hand, they are uniformly bounded in $W^{-1,p}_\loc$ for $p>n\geq 2$.
Thus, by interpolation, we find that
\begin{equation*}
\{\di V^{(B^\e)}, d\Omega^{(B^\epsilon)}, \di V^{(\na^{\perp,\e})}, d\Omega^{(\na^{\perp,\e})}\} \quad\,
\text{ are pre-compact in } H^{-1}_\loc,
\end{equation*}
which is precisely the desired first-order differential constraints
for the geometrically intrinsic div-curl lemma, Theorem \ref{thm: geometric div-curl}.

Therefore, applying Theorem \ref{thm: geometric div-curl},
we obtain the following subsequential convergence results in $\mathcal{D}'(M)$:
\begin{eqnarray*}
&&\langle \vbe_{W,\eta}(X,Y), \obe_{Z,\eta}\rangle \longrightarrow  \langle \vb_{W,\eta}(X,Y), \ob_{Z,\eta}\rangle,\\[1mm]
&&\langle\vnabe_{\eta,\beta}(X,Y),\onabe_{\xi,\beta}\rangle \longrightarrow\langle\vnab_{\eta,\beta}(X,Y),\onab_{\xi,\beta}\rangle,\\[1mm]
&&\langle \vbe_{Z,\xi}(X,Y),\obe_{Z,\eta} \rangle \longrightarrow \langle \vb_{Z,\xi}(X,Y),\ob_{Z,\eta} \rangle,\\[1mm]
&&\langle\vnabe_{\eta,\beta}(X,Y),\obe_{Z,\beta}\rangle \longrightarrow\langle\vnab_{\eta,\beta}(X,Y),\ob_{Z,\beta}\rangle,
\end{eqnarray*}
so that we can pass to the limits in Eqs. \eqref{eqn_Gauss Eqn in second computation}--\eqref{eqn_Ricci Eqn in second computation}.
As shown in Step $2$,
these equations are equivalent to the GCR equations, which leads to the weak continuity of the GCR equations.
Finally, by Step $1$, we know that the existence of solutions of the GCR equations in $\mathcal{D}'(M)$ are equivalent to the existence
of isometric immersions in $W^{2,p}_\loc(M;\rnk)$.
Thus, the assertion is proved on the local trivialized chart $U \subset M$ as in Step $1$.
When $M$ is simply-connected, we can pass from the local to global by a standard monodromy argument.
\end{proof}

We now make three comments on our main theorem, Theorem \ref{thm: weak rigidity}.

First of all, in Step $1$, we have given a geometrically intrinsic proof
of the {\it realization theorem}.
This can be summarized as follows:

\begin{corollary}[Theorem 5.2 in \cite{chenli}]\label{cor_equivalence of 3 formulations}
Let $(M,g)$ be an $n$-dimensional, simply-connected Riemannian manifold
with metric $g\in\woneploc$ for $p>n$,
and let $(E,M,\real^k)$ be a vector bundle over $M$.
Assume that $E$ has a $\woneploc$ metric $g^E$ and an $\lploc$ connection $\na^E$
such that  $\na^E$ is compatible with  $g^E$.
Moreover, suppose that there is
an $\lploc$ tensor field $S:\Gamma(E) \times \vf \rightarrow \vf$ satisfying
\begin{equation*}
\langle X, S_\eta (Y)\rangle - \langle S_\eta (X), Y\rangle =0,
\end{equation*}
and a corresponding $\lploc$ tensor field $B: \vf \times \vf \rightarrow \Gamma (E)$ defined by
\begin{equation*}
\langle B(X,Y),\eta \rangle = -\langle S_\eta (X), Y \rangle.
\end{equation*}
Then the following are equivalent{\rm :}
\begin{enumerate}
\item[\rm (i)]
The GCR equations as in Proposition {\rm \ref{propn: GCR equations}} with $R^\perp$ replaced by $R^E$,
the Riemann curvature operator on the bundle{\rm ;}
\item[\rm (ii)]
The Cartan formalism{\rm ;}
\item[\rm (iii)]
The existence of a global isometric immersion $f\in W^{2,p}_{\rm loc}(M; \real^{n+k})$ such that the induced normal
bundle $T(fM)^\perp$, normal connection $\na^\perp$, and second fundamental form can be identified
with $E, \na^E$, and $B$, respectively.
\end{enumerate}
In {\rm (i)}--{\rm (ii)}, the equalities are taken in the distributional sense and, in {\rm (iii)},
the isometric immersion $f\in W^{2,p}_{\rm loc}$ is unique {\it a.e.},
modulo the Euclidean group of rigid motions $\real^{n+k} \rtimes O(n+k)$.
\end{corollary}
In view of the above corollary, for the purpose of weak rigidity,
it is more natural to investigate the Cartan formalism.
In particular, the GCR equations are recast into the compact identity $dW = W\wedge W$,
which is the second structural equation (as in well-known in geometry, the first structural equation expresses
the torsion-free property of the Levi-Civita connection).
However, notice that the connection $1$-form $W$ consists of only $(B,\na^\perp)$
so that, for the sequence of isometric immersions $\{f^\e\}$ uniformly bounded in $W^{2,p}_\loc(M;\rnk)$,
the corresponding $\{W^\e\}$ is uniformly bounded in $L^p_\loc(M; \mathfrak{so}(n+k))$.
Let $\bar{W}$ be a weak limit of $\{W^\e\}$.
Then, via similar arguments as in Steps 2--3 in the proof of Theorem \ref{thm: weak rigidity},
we can pass the limits in
\begin{equation}
dW^\e = W^\e \wedge W^\e\quad \Longrightarrow\quad d\bar{W} = \bar{W} \wedge \bar{W}.
\end{equation}
Therefore,
applying Corollary \ref{cor_equivalence of 3 formulations},
we obtain a simplified proof of Theorem \ref{thm: weak rigidity}.

Second, for the most {\em physically relevant} case of the isometric immersing/embedding
of a $2$-dimensional manifold ({\it i.e.}, a surface)
into $\R^3$, we can also establish the weak rigidity of the GCR equations in the critical case $p=n=2$
(where the Ricci equation is trivial).
This is because, on the right-hand side of the Gauss equation \eqref{eqn_gauss},
we have the Gauss curvature $R(X,Y,Z,W)$, which is a {\em fixed $L^1$ function}
in the setting of isometric immersions,
thanks to the Cauchy-Schwarz inequality.
Therefore, it is equi-integrable. Then we can apply the critical case of the div-curl
lemma (Theorem \ref{thm_ generalised critical case, div curl lemma}).

\begin{corollary}[Theorem 6.3 in \cite{chenli}]\label{cor:endpoint}
Let $M$ be a $2$-dimensional, simply-connected surface,
and let $g$ be a metric in $H^{1}_{\rm loc}$. If $\{f^\epsilon\}$ is a family of $H^{2}_{\rm loc}$ isometric
immersions of $M$ into $\real^3$ such that the corresponding second fundamental
forms $\{B^\epsilon\}$ are uniformly $L^2$-bounded.
Then, after passing to a subsequence, $\{f^\epsilon\}$ converges
to $\bar{f}$ weakly in $H^2_{\rm loc}$ which is still an isometric immersion $\bar{f}:(M,g) \rightarrow \real^3$.
Moreover, the corresponding second fundamental form $\bar{B}$ is a limit point of $\{\bep\}$ in the $L^2_{\rm loc}$ topology.
\end{corollary}

Finally, it is easy to derive a slightly more general version of the weak rigidity theorem,
Theorem \ref{thm: weak rigidity}, by allowing the metrics to be unfixed and strongly convergent
in $W^{1,p}_\loc$ for $p>n$.
Such scenarios naturally arise in the regularization of a singular metric into
smooth ones; {\it cf}. \S 7.2 in \cite{chenli}.

We remark in passing that the analogies of Theorem \ref{thm: weak rigidity}
and Corollaries \ref{cor_equivalence of 3 formulations}--\ref{cor:endpoint}
for isometric immersions into semi-Euclidean spaces of semi-Riemannian submanifolds
({\it i.e.}, the metrics are non-degenerate, but may no longer be positive-definite; see O'Neill \cite{oneill}) are also valid.
For the possibly degenerate hypersurfaces, using the machinery of rigging fields ({\it cf.} \cite{schouten, lefloch2, mars}),
a counterpart of the Cartan formalism can be established,
which leads to the weak rigidity, provided that the rigging fields
are uniformly $L^p_\loc$ bounded. For a rigorous formulation and the proof of these
results, see our forthcoming paper \cite{chenli-2}.

As discussed above, we have established the weak rigidity of isometric
immersions (Theorem \ref{thm: weak rigidity}) in \cite{chenli}.
It would be interesting
to explore its relation with the rigidity/non-rigidity results in stronger
topologies (see the discussion at the beginning of \S 4),
to extend it to the larger framework of the h-principle laid down by Gromov (\cite{gromov}),
and to examine what the possible implications are in fluid dynamics, in view of the connections
between isometric immersions and  Euler equations (see \cite{ACSW,CSW10-CMP} and \cite{delellis}).

\bigskip
\noindent
{\bf Acknowledgement}.
Gui-Qiang Chen's research was supported in part by
the UK EPSRC Science and Innovation Award
to the Oxford Centre for Nonlinear PDE (EP/E035027/1),
the UK EPSRC Award to the EPSRC Centre for Doctoral Training
in PDEs (EP/L015811/1), and the Royal Society--Wolfson Research Merit Award (UK).
Siran Li's research was supported in part by the UK EPSRC Science and Innovation Award
to the Oxford Centre for Nonlinear PDE (EP/E035027/1).

\end{document}